\documentclass[12pt,leqno, english]{amsart}
\title[Pullbacks of groups]
{Bredon Cohomology, $K$-theory and $K$-Homology  of Pullbacks of  groups.}
\keywords{Eilenberg-Moore Spectral  Sequence, Bredon Cohomology, Negative algebraic  K-theory, Chrystallographic  groups. 2010 Math  Subject  Classification: Primary: 19K99, Secondary:   19K33, 55N15, 55N91, }
\author{No\'{e} B\'{a}rcenas }
        \address{
Centro  de  Ciencias Matem\'aticas. \\ UNAM Campus  Morelia\\Ap.Postal  61-3 Xangari\\ Morelia, Michoac\'an.  M\'EXICO 58089}
         \email{barcenas@matmor.unam.mx}
         \urladdr{http://www.matmor.unam.mx/~barcenas}
         
\author{Daniel Juan-Pineda}
\address{Centro  de  Ciencias Matem\'aticas. \\ UNAM Campus  Morelia\\Ap.Postal  61-3 Xangari\\ Morelia, Michoac\'an.  M\'EXICO 58089}
 \email{daniel@matmor.unam.mx}

\author{Mario Vel\'asquez}
\address{Centro  de  Ciencias Matem\'aticas. \\ UNAM Campus  Morelia\\Ap.Postal  61-3 Xangari\\ Morelia, Michoac\'an.  M\'EXICO 58089}
 \email{mavelasquezm@gmail.com}
 \urladdr{https://sites.google.com/site/mavelasquezm/}
         \date{\today}

\usepackage{amsmath}
\usepackage{hyperref}
\usepackage{verbatim}
\usepackage{pdfsync}
\usepackage{calc}
\usepackage{enumerate,amssymb}
\usepackage[arrow,curve,matrix,tips,2cell]{xy}
  \SelectTips{eu}{10} \UseTips
  \UseAllTwocells
\usepackage{graphicx}
\usepackage{pb-diagram}

\DeclareMathAlphabet\EuR{U}{eur}{m}{n}
\SetMathAlphabet\EuR{bold}{U}{eur}{b}{n}

\makeindex             

\usepackage{pstricks}
\usepackage{pst-3dplot}

\theoremstyle{plain}
\newtheorem{theorem}{Theorem}[section]
\newtheorem{lemma}[theorem]{Lemma}
\newtheorem{proposition}[theorem]{Proposition}
\newtheorem{corollary}[theorem]{Corollary}

\theoremstyle{definition}
\newtheorem{definition}[theorem]{Definition}
\newtheorem{example}[theorem]{Example}
\newtheorem{condition}[theorem]{Condition}
\newtheorem{remark}[theorem]{Remark}

\newtheorem{convention}[theorem]{Convention}
\newtheorem*{theoremn}{Theorem}

{\catcode`@=11\global\let\c@equation=\c@theorem}

\newcommand{\comsquare}[8]                   
{\begin{CD}
#1 @>#2>> #3\\
@V{#4}VV @V{#5}VV\\
#6 @>#7>> #8
\end{CD}
}

\newcommand{\xycomsquare}[8]                   
{\xymatrix
{#1 \ar[r]^{#2} \ar[d]^{#4} &
#3 \ar[d]^{#5}  \\
#6\ar[r]^{#7} &
#8
}
}

\newcommand{\calf}{\mathcal{F}}

\newcommand{\calfin}{\mathcal{FIN}}

\newcommand{ \calr}{\mathcal{R}}

\newcommand{\IC}{{\mathbb C}}

\newcommand{\IH}{{\mathbb H}}

\newcommand{\IK}{{\mathbb K}}

\newcommand{\IQ}{{\mathbb Q}}
\newcommand{\IR}{{\mathbb R}}

\newcommand{\IZ}{{\mathbb Z}}
\newcommand{\IZc}{{\mathbb Z}/{4\mathbb{Z}}}

\newcommand{\curs}{\EuR}

\newcommand{\CHAINCOMPLEXES}{\curs{CHCOM}}

\newcommand{\Or}{\curs{Or}}
\newcommand{\RINGS}{\curs{RINGS}}

\newcommand{\Hom}{\operatorname{Hom}}

\newcommand{\Tor}{\operatorname{Tor}}

\newcommand{\EGF}[2]{E_{#2}(#1)}                   
\newcommand{\eub}[1]{\underline{E}#1}              

\newcommand{\higherlim}[3]{{\setbox1=\hbox{\rm lim}
        \setbox2=\hbox to \wd1{\leftarrowfill} \ht2=0pt \dp2=-1pt
        \mathop{\vtop{\baselineskip=5pt\box1\box2}}
        _{#1}}^{#2}#3}

\newcommand{\version}[1]                       
{\begin{center} last edited on #1\\
last compiled on \today \\
name of texfile: \jobname
\end{center}
}

\newcounter{commentcounter}

\makeindex

\begin{document}

\begin{abstract}

We  develop a  spectral  sequence of  Eilenberg-Moore type  to  compute  Bredon  Cohomology  of  spaces  with  an action  of a  group  given  as a  pullback. 

Using several other  spectral  sequences, and  positive  results  on the Baum-Connes  Conjecture,  we  are  able  to  compute  Equivariant  $K$-Theory  and  $K$-homology  of the reduced  $C^*$-algebra of  a 6-dimensional crystallographic  group $\Gamma$ introduced  by  Vafa  and  Witten.  
 We  also  use  positive  results   on the  Farrell-Jones conjecture to  give  a  vanishing result  for the algebraic  $K$-theory  of  the  group  ring of  the  group $\Gamma$ in negative degrees. 
\end{abstract}

\maketitle

\section{Introduction}\label{Introduction}

We  develop a  method  to  compute  Bredon  Cohomology  and  equivariant (co)-homology  theories  of  spaces  with an  action  of  a  discrete  group $\Gamma$   obtained  as  a  pullback of discrete  groups over a  finite  group. 

 \begin{condition}\label{diagrampullback}
 Let   $\Gamma$ be  a  group which   is obtained as a pullback diagram
$$
\xymatrix{
  \Gamma \ar[r]^{p_2}\ar[d]_{p_1} & H\ar[d]^{\pi_2}\\
  G \ar[r]_{\pi_1} & K}
$$
where $K$  is  a  finite  group. 
  \end{condition}

Given  such a  pullback  diagram,  the group $\Gamma$ can be viewed as a subgroup of $G\times H$, namely $\Gamma=\{(g,h)\in G\times H\mid \pi_1(g)=\pi_2(h)\}$.  When the maps in the pullback are clear from the context, we denote this pullback by $G\times_KH$.
We develop a spectral sequence in Theorem \ref{eilenbergmoore} converging to Bredon cohomology groups of spaces  with a    $\Gamma$-action when $\Gamma$ is defined as a pullback as in condition \ref{diagrampullback}.

\begin{theoremn} \ref{eilenbergmoore} [Eilenberg-Moore spectral sequence for  Bredon cohomology]
  Let  $\Gamma$  be  a  group   as in \ref{diagrampullback}. Assume  that  $X$ is  a   $G$-CW  complex with  isotropy  in a family $\calf$   of  finite  subgroups  of $G$ and $Y$ is  an $H$-CW  complex with  isotropy  in a  family  $\calf^{'}$ of  finite  subgroups  of $H$.  Let  $M(?)$ be  a  Bredon  coefficient  system, which takes values in the category  of commutative rings. Assume  that  $M(?)$  satisfies  the  projectivity  condition \ref{conditionP}. Then, there  is  a  spectral  sequence  with  $E_2$  term  given  by  
  $$
   \Tor_{M(K/K)}^{p,q}(H_\calf^*(X,M),H_{\calf'}^*(Y,M))  
   $$ 
  which  converges  to 
  $$
  \Tor_{M(K/K)}^{p,q}(\underbar{C}_G^*(X,M),
  \underbar{C}_H^*(Y,M)).
  $$
  \end{theoremn}
The  groups $H_{\calf'}^*(Y,M)$ denote  Bredon  co-homology  with coefficients  in   the  contravariant  functor $M$,  defined  on  a  family  $\calf'$  containing  the  isotropy groups  of $Y$. The  groups  $\underbar{C}^*_G(X,M)$  denote   Bredon  cochain  complexes  with  a differential  graded  structure explained  in detail  in Section \ref{sectiondga}. The groups $\Tor_{M(K/K)}^{p,q}(\;, \;)$ are   derived  functors  of  differential  graded  algebras  and  modules. 

 The  spectral  sequence  gives a method  to  compute $\Gamma$-equivariant Bredon cohomology groups   out  of the  (potentially  easier  to  calculate) $G$- respectively, $H$-equivariant  cohomology  groups  of  $X$ and $Y$,  together  with knowledge  about   their  structure  as modules over  the  ring $M(K/K)$. 

The pullback  structure  in  \ref{diagrampullback} appears  in  the  computations  of  Bredon  cohomology  of  crystallographic  groups $\Gamma$ with a  given (finite) point  group $K$. 
These  groups  are  given as  an  extension  
\begin{equation}\label{equationcristal}
1\to  \mathbb{Z}^n\to \Gamma\to K\to  1
\end{equation}
where $K$ is finite and the conjugation action  on $\mathbb{Z}^n$ is given by a  representation $\rho:K\to Gl_n(\IZ)$. 
IN  this  situation,  the  space $\mathbb{R}^{n} $ with the  induced  action  is a  model  for  $\eub{\Gamma}$. This  is a consequence  of  Proposition  1.12, page 30 in  \cite{connolykosniewsky}.

Splitting  the  representation  $\rho: K \to  Gl_n(\IZ)$  gives  a  pullback  structure  on  $\Gamma$. More  precisely, assume  that $\IZ^n$  with  the  action  given  by  $\rho$  has  a  $K$-invariant  decomposition $\IZ^n[\rho]= A\oplus B$. 

Denote  by  $G$  the  semidirect  product  $A\rtimes  K$  and  by  $H$  the  semidirect  product  $B\rtimes  K$. Then, the group $\Gamma$  is  isomorphic  to  the  pullback  $G\times_K H$. See Remark  \ref{remarkcrystallographic}  for  details  on  this.

The  main  application  of  Theorem \ref{eilenbergmoore}  will be a  method  for  the  computation of  equivariant (co)-homology  theories   evaluated  on classifying  spaces  for  families  of subgroups of $\Gamma$  as  in \ref{cristalograficos}.

 The  interest  in these  computations  comes  from the  fact  that  the  assembly  maps in the   Baum-Connes  conjecture \cite{baumconnes}
\begin{equation}\label{equationbaumconnes}
K_*^ \Gamma(\eub{\Gamma})\to K_*(C^*_r(\Gamma))
\end{equation}
and in  the  Farrell-Jones  Conjecture \cite{luck1998}
\begin{equation}\label{equationfarrelljones}
\IH_*^\Gamma(\EGF{\Gamma}{\mathcal{VC}}, \IK^{-\infty}(R))\to K_*(R\Gamma)
\end{equation}
involve   equivariant  homology  theories  evaluated  on these  spaces. Computations  of  Bredon  (co)-homology  groups associated  to  these (co)-\-ho\-mology  theories   give inputs  to a  spectral  sequence  of  Atiyah-Hirze\-bruch  type  \cite{luck1998},  abutting  to the   relevant  equivariant  (co)-\-ho\-mo\-logy groups.

Until  now, the  methods   developed  for the  computation   of $K$-theory  and  $K$-homology  groups  of extensions    $\Gamma$ as   in \ref{equationcristal} include  assumptions on   the maximality  of finite, respectively virtually cyclic  subgroups  relevant  to  the  computation, as  well  as strong  hypotheses  on  their  normalizers. This concerns particularly  conditions M  and NM  in   \cite{luck-stamm}, \cite{luckdavissemidirectproduct},   or  explicit  computations  related  to the  Weyl groups  of  them, as  in \cite{luecklangertopologicalktheory}.  
All  of  them  restrict  the  class  of  extensions  to  those  arising  from   conjugation  actions which are  free  outside  of the  origin.

In  another  direction,  extensive  knowledge  of  models for  both  spaces,       using  the  classification  of  crystallographic groups  in a  given  dimension  also  gives information  about the  homology  groups  relevant  to the Farrell-Jones  Conjecture, as  it  is  done in  \cite{farleyortiz}. 

The methods  derived  from the  spectral  sequence  in Theorem \ref{eilenbergmoore} rely neither on  the  dimension, as  the  use  of  specific models  in \cite{farleyortiz},  nor  on freeness of  the  conjugation  action as  in \cite{luecklangertopologicalktheory}, \cite{luck-stamm}, \cite{luckdavissemidirectproduct}.

 To  illustrate  our  method,  we  concentrate  in  a  group  extension 
 $$
 1\to \mathbb{Z}^6\to \Gamma\to \mathbb{Z}/4\IZ\to 1 ,
 $$
 which  gained  interest  in    theoretical  physics \cite{vafa-witten}.

  \begin{example}\label{example}[The 6- dimensional Vafa-Witten toroidal  orbifold  quotient]
  Consider the action of $\IZ/4\IZ$ on $\IZ^{\oplus6}$ induced from the 
  action of $\IZ/4\IZ$ on $\IC^3$, given by
  $$
  k(z_1,z_2,z_3)=(-z_1,iz_2,iz_3).
  $$
  The associated semidirect product 
  \begin{equation}\label{vafawittengroup}
  1\to \mathbb{Z}^6 \to \Gamma \to \IZ/4\IZ \to 1
  \end{equation}
  will  be  called  the  Vafa-Witten  group, this splits as a multiple pullback
  $$
  (\IZ\rtimes\IZ/4\IZ)\times_{\IZ/4\IZ}(\IZ\rtimes\IZ/4\IZ)\times_{\IZ/4\IZ}(\IZ^2\rtimes\IZ/4\IZ)\times_{\IZ/4\IZ}(\IZ^2\rtimes\IZ/4\IZ).
  $$
  Where the first two semidirect products are taken with respect to  the  conjugation action of $\mathbb{Z}/4$  on  $\mathbb{Z}^2$  given  by  scalar  multiplication by  $-1$,  and the two  last  ones  are  given  by  the conjugation action  on $\mathbb{Z}^2$ given by complex  multiplication by   $i$. Notice  that  the  conjugation action determined  by  $\Gamma$, and  more  specifically,  the one  coming  from the block given  by  the  action $\mathbb{Z}\rtimes\mathbb{Z}/4\mathbb{Z}$ is  not free  outside  of the  origin. Condition $NM$  of  \cite{luck2005} is  not  satisfied  in  this  case, although  our  methods  readily  apply  to this  situation.  
  \end{example}

For  the  group  described  in example \ref{example},  we show  that  the spectral sequence from  Theorem \ref{eilenbergmoore} collapses at the  $E_2$-term for  the  specific  choice  of the complex  representation  ring  as a  Bredon  coefficient system. With the  use  of  a  Universal Coefficient  Theorem  for  Bredon  Cohomology, Theorem 1.13  in \cite{BAVE2013}, 
    completely determines  the equivariant  $K$-Homology  of the  classifying  space  for  proper actions. 

\begin{theoremn}\ref{cristalograficos} [Topological $K$-Theory]
  Let $\Gamma$  be  the group $\IZ^6\rtimes\IZc$ acting on $\IR^6$ as in \ref{example}. The topological  $K$-theory  of  the  reduced  $C^*$-algebra  of  $ \Gamma$  is  as  follows: 
    \begin{itemize}
    \item $K_0(C_r^*(\Gamma))\cong K_0^{\Gamma}(\underbar{E}\Gamma)\cong\IZ^{\oplus47}$ and
    \item $K_1(C_r^*(\Gamma))\cong K_1^{\Gamma}(\underbar{E}\Gamma)=0$.
    \end{itemize}
  \end{theoremn}

The  ideas  developed  in Theorem \ref{eilenbergmoore} and   subsequent computations are  particularly  well-suited  to  families  of  subgroups  which  are well-behaved  under  products. The  example  for  such a  family  is, notably,  the  family of  finite   subgroups. Although  the  Eilenberg-Moore  method \ref{eilenbergmoore}  does  not  transfer directly to  the  family  of  virtually  cyclic  subgroups due  to  its  bad  behaviour  under  products,   we  are  able  to  deduce using positive results on  the Farrell-Jones  Conjecture \cite{tsapogas}, \cite{Dj}, \cite{farrelvirtuallycyclic} and  computations  of  lower  algebraic  $K$-Theory \cite{carterlocalization}, \cite{carterlower},  the  following result, computing  the  negative  algebraic $K$-theory of   the  group ring $RG$.

\begin{theoremn}\ref{theoremnegative} [Negative algebraic  K-Theory]
Let  $\Gamma$  be the  group  determined  by the  extension  \ref{vafawittengroup}. Let  $R$ be a ring  of algebraic  integers.  Then,
$$ 
K_i(R\Gamma)=0, \text{ for all }i<0.
$$
\end{theoremn}
  
This  paper  is  organized  as  follows: 

\tableofcontents

\subsection*{Acknowledgements}
The  first  author received   support  of a  CONACYT Postdoctoral  Fellowship. The second  author  received  support  from   DGAPA  and  CONACYT research grants.  The  third  author  received support  of  a  UNAM Postdoctoral  Fellowship.
 
 We thank Bernardo Uribe for valuable remarks in a preliminary version of this paper.

\section{The  Eilenberg-Moore  spectral  sequence and  Bredon Cohomology}\label{sectiondga}

  \begin{definition}
  Recall  that  a  $G$-CW  complex structure  on  the  pair $(X,A)$  consists  of a  
  filtration of  the $G$-space $X=\cup_{-1\leq n } X_{n}$ with $X_{-1}=\emptyset$,  
  $X_{0}=A$ and every  space   is inductively  obtained  from  the  previous  one   
  by  attaching  cells  with  pushout  diagrams  
  $$\xymatrix{\coprod_{\lambda} S^{n-1}\times G/H_{\lambda} \ar[r] \ar[d] & X_{n-1} \ar[d] \\ 
  \coprod_{\lambda}D^{n}\times G/H_{\lambda} \ar[r]& X_{n}}$$   
  \end{definition}

\begin{definition}
Let  $\calf$  be   a  family of  subgroups  which  is  closed  under  subgroups  and  conjugation. A model  for  the  classifying  space for  the  family  $\mathcal{F}$  is  a $G$-CW  complex $X$  satisfying  
\begin{itemize}
  \item{All  isotropy  groups of  $X$  lie  in $\mathcal{F}$.}
\item{For  any $G$-CW  complex $Y$ with  isotropy  in $\mathcal{F}$, there  exists up  to  $G$-homotopy   a  unique  $G$-equivariant  map $f:Y\to X$. }
\end{itemize}
\end{definition}
A  model  for  the classifying  space  of  the  family  $\mathcal{F}$  will be  usually  denoted  by  $ \EGF{G}{\mathcal{F}}$. 

Particularly  relevant   is  the classifying  space  for  proper actions,  the  classifying  space  for   the  family $\calfin$  of  finite  subgroups, denoted  by  $\eub{G}$ and  the  space  $\EGF{G}{\mathcal{VC}}$  for  the  family $\mathcal{VC}$ of  virtually cyclic subgroups.

 Let  $X$  be  a  $G$-CW-complex. The Bredon  chain complex is  defined  as the contravariant   functor  to the  category of  chain complexes
  ${C}_{*}^G(X):\Or_\calf(G)\to  \mathbb{Z}-\CHAINCOMPLEXES  $ which assigns  
  to every  object  
  $G/H$     the  cellular $\mathbb{Z}$-chain  complex   of  the $H$-fixed point  
  complex    ${C}_{*}(X^ {H})\cong C_{*}({\rm  Map  }_{G}(G/H, X))$  with  
  respect  to  the  cellular  boundary  maps $\underline{\partial}_{*} $. The $n$-chains  of  the  Bredon  chain  complex   evaluated  on an  object  $G/K$ of  the  orbit category, consist  of elements  of  free  abelian  groups $\bigoplus_\lambda \IZ[e_\lambda$,  where  $e_\lambda$ denote the cell  orbits   of  type $D^n\times G/K$ in  the  cell  decomposition  above.

 Let $G$  be  a  discrete  group,  let  $\Or(G)$ be  the  orbit  category of  $G$,  where  objects  are  homogeneous sets $G/H$  and   morphisms  are  $G$-equivariant  maps. 
   
Let  $R$  be  a   ring. Recall  that  a  contravariant  Bredon  functor $M$ with  values  on  $R$-modules  is  a  contravariant  functor   defined  on $\Or(G)$ to  the  category  of $R$-modules.   

 \begin{definition}[Bredon  co\-cha\-in complex]
 
Given  a  contravariant  Bredon  functor $M$, the  Bredon    cochain   complex  $C_G^*(X;M)$ is  defined  as the   abelian  group   of  natural  transformations   of  functors  defined  on  the  orbit  category ${C}^{*}_G(X) \to  M$. In  symbols, 
$$
C_G^n(X;M)=\Hom_{\Or_{\calf(G)}}({C}_n(X),M),
$$
where  $\mathcal{F}(G)$ is  a  family containing  the  isotropy  groups  of  $X$. 

Given a  set  $\{e_{\lambda}\}$ of   representatives of  the  orbits  of n-cells of  the  $G$-CW  complex  $X$,  and isotropy  subgroups  $P_{\lambda}$  of  the  cells  $e_{\lambda}$,   the  abelian  groups $C_G^n(X,M)$  satisfy:
 $$
 C_G^n(X,M)= \underset{\lambda}{\prod }Hom_{\mathbb{Z}}(\mathbb{Z}[e_{\lambda}], M(G/P_{\lambda}))
 $$   
 with  one  summand  for  each  orbit representative  $e_\lambda$.
 They  afford  a differential $\delta^n:C_G^n(X,M)\to C_G^{n+1}(X,M)$ determined  by  $\underline{\partial}_*$ and  maps $M(\phi):  M(G/P_\xi)\to M( G/ P_\lambda )$  for  morphisms  $\phi:G/P_\lambda \to  G/P_\xi$.

Given  a  functor  $M$ taking  values  in  the  category  of commutative  rings  with  1, the  Bredon  cochain  complex  has  cup  products
$$
\cup: C_G^m(X,M)\otimes C_G^n(X,M)\to C_G^{n+m}(X,M).
$$
See \cite{bredon}, Chapter I.8  in pages  19-20.

We  will  list now  some  algebraic  definitions. 

\begin{definition}[Differential  Graded  Algebra]
Let  $A$  be  a  graded  algebra. $A$  is  said  to  be  a  differential  graded  algebra  if  there  exists a  group homomorphism $d:A\rightarrow A$ of degree $+1$ satisfying
\begin{enumerate}
\item $d^2=0$
\item $d(ab)= d(a)b + (-1)^{\mid a\mid} ad(b)$
where $\mid a\mid$  is the  degree of  the  element  $a\in A$. 
\end{enumerate}

\end{definition}

\begin{definition}[Differential Modules  over  a  Differential  Graded  Algebra]

Let  $(A,d_A)$ be  a differential  graded  algebra. A  differential  graded  module over $A$  is a  graded $A$-module  $M$  together  with  differentials $d_M$ satisfying   $d_M(am)= d_A(a)m + (-1)^{\mid m\mid} ad_M(m)$

\end{definition}

\begin{remark}[DGA Associated  to a  Bredon Module]
Let  $M$  be  a  contravariant  functor  defined  on  the  full  subcategory  $\Or(G, \mathcal{F})$ consisting  of  homogeneneous  spaces $G/H$,  where  $H\in \mathcal{F}$. Assume  $M$ takes  values  in   the  category  of commutative rings with  1. The  differential  graded  algebra $C^*_G(M)$  is  defined  as the  inverse  limit 
$$
C^0_G(M)= \underset{G/P \in \mathcal{F}}{\lim}M(G/P)
$$
where  the  limit is  taken in the  category  of  commutative rings with 1,  
$C^i_G(M)=0$  for  $i\neq0$, and  $d_i=0$  for  all $i$. Note that if the group $G$ is finite, and $\calf$  is  the  family  of finite  subgroups,  $C^0_G(M)=M(G/G)$.
\end{remark}

The  full  Bredon  cochain  complex  $\bigoplus_n C_G^n(X,M)$ together  with  the differential  graded  $C_G^*(M)$-module  structure will be  denoted by ${C}_G^*(X,M)$.        
\end{definition}

 \begin{definition}[Bredon cohomology] 
  The  Bredon  cohomology  groups   with  coefficients  in  $M$, denoted  by  
  $H^{*}_{G} (X,  M)$    are  the  cohomology   groups  of  the  cochain  complex  
  $\big ({C}_{G}^ *(X, M), \delta^* \big )$. 
  \end{definition}

We   will  now  assume  the  following  condition,  which  simplifies  the  differential  graded  structure  involved in  the  cochain  complexes.

\begin{condition}[Condition P]\label{conditionP}
  We  will  assume that $\Gamma$  fits  in a  pullback  diagram as  in condition \ref{diagrampullback}.
  
  
Consider a contravariant Bredon functor $M(?)$ taking  values  on the category  of commutative rings with  $1$.  Let $P$ be  a  finite  subgroup  of  $\Gamma= G\times_K H$. 
\begin{itemize}
\item  The  maps  $(\pi_1\circ p_1)^*: M(K/\pi_1(P))\to M(\Gamma/P)$, $(\pi_2\circ p_2)^*:M(K/\pi_2(P))\to  M(\Gamma/P)$ furnish  $M(\Gamma/P)$ with  a  structure  of a projective module  over the  ring   $M(K/\pi_1(P))\cong  M(K/\pi_2(P))$. 
\end{itemize}
\end{condition}


The  following  fact  is  crucial  for  our  computations  related  to  the  family  of  finite  subgroups.

\begin{lemma}[Structure Lemma  for  families  of  finite  subgroups] \label{lemmafamilyfiniteisotropies}
Let  $\Gamma$  be  a  group  given as  a pullback as  in \ref{diagrampullback}.   
\begin{itemize}

\item The  structure  maps  $p_1$  and  $p_2$ give  a bijective  correspondence  between  the  elements  of the  family $ \calfin(\Gamma)$  and  the  family $\calfin(G)\times_K \calfin(H)
 :=\{(P\times_{\pi_1(P)} Q)\mid P\in\calfin(G), Q\in\calfin(H)\}$.
 
\item Let   $X$ and $Y$ be  proper  $G$-,  respectively  $H$-CW-complexes. Then,  the isotropy groups of the action of $G\times_KH$ in $X\times  Y$ are  contained in  the  family $\calfin(G)\times_K\calfin(H)$. 

\end{itemize}

\end{lemma}

  \begin{proposition}\label{pullbackrepresentations} Given a pullback diagram as in condition \ref{diagrampullback}, and a restriction of the pullback to   finite  subgroups
  $$
  \xymatrix{
  \Gamma_1 \ar[r]^{p_2}\ar[d]_{p_1} & Q\ar[d]^{\pi_2}\\
  P \ar[r]^{\pi_1} & K_1}  
  $$
  there is a natural isomorphism of 
  $M(K/\pi_1(P))$-modules 
  $$
  M(\Gamma/\Gamma_1)\cong M(G/P)\otimes_{M(K/\pi_1(P))}M(H/Q).
  $$
  Where the $M(K/\pi_1(P))$-module structure in both sides is given by the pullback 
  diagram.
  \end{proposition}
\begin{proof}
  If $M$ takes values on the category of  commutative  rings  with 1, then, the  ring  homomorphisms  $p_2^* $ and $p_1^*$ give  a  map 
$M(G/P)\otimes_\IZ M(H/Q)\to M(\Gamma/\Gamma_1)$,  which  defines  an isomorphism  \\  $M(G/P)\otimes_{M(K/\pi_1(P))} M(H/Q)\to M(\Gamma/\Gamma_1)$
\end{proof}

\begin{definition}

Let  $\Gamma  $  be a  group  given  as a  pullback  as  in Condition \ref{diagrampullback}. Let  $M$  be  a  contravariant  Bredon  functor  taking  value  on  the  category  of  commutative  rings  with  1.
We  will  denote  by $M^G(?)$,  respectively $M^H(?)$  the  functors $p_1^*(M)$, respectively  $p_2^ *(M)$.   
Consider  the  category  $\Or(G, \calfin)\times \Or(H,\calfin)$ and  consider  the  functor  defined  on  objects  $G/R\times H/Q$  as  $M(G/R)\otimes_\IZ M(H/Q)$.  

We  will  denote the restriction  of  this  functor to   $\Or(G\times_KH, \calfin)$ by $M^G(?)\otimes_\IZ M^H(?)$. On  each object $G\times_K H/ P\times _{\pi_1(P)}Q $, 

$$M^G(?)\otimes_\IZ M^H(?):Or_{\calf\times_K\calf'}(G\times_KH)\rightarrow \RINGS$$
$$(G\times_KH)/(P\times_{\pi_1(P)}Q)\mapsto M(G/P)\otimes_{\IZ} M(H/Q).$$
\end{definition}

\begin{convention}
We can define a further equivalence relation over the restriction of this tensor product, we say $\alpha\cdot\rho_1\otimes\rho_2\sim\rho_1\otimes\alpha\cdot\rho_2$, where $\alpha\in M(G/\pi_1(P))$ and the products in both sides are defined via the  maps $\pi_1^*$ and $\pi_2^*$. Denote the quotient by $(M^G\otimes_{M^K}M^H)(?)$. 
\end{convention}
\begin{lemma}\label{lemmanaturalequivalencetensor}
The isomorphism in Proposition \ref{pullbackrepresentations} can be promoted to a natural equivalence between the functors $M^\Gamma(?)$ and $(M^G\otimes_{M^K}M^H)(?)$. 

\end{lemma}

\begin{proof}
Given a $G$-map
  $$G\times_KH/P\times_{\pi_1(P)}Q\rightarrow 
  G\times_KH/P'\times_{\pi_1(P')}Q', $$  this map is characterized by an element 
  in   $G\times_KH $ that conjugates $P\times_{\pi_1(P)}Q$ to a subgroup of 
  $P'\times_{\pi_1(P')}Q'$. 
  Now,  taking  the 
  restriction to subconjugate subgroups commutes  with  taking the pullbacks with  respect to $p_1$    and $p_2$ due  to the  structure  lemma \ref{lemmafamilyfiniteisotropies}.
\end{proof}

  Taking  the associated  differential  algebra  structure,  one obtains:

 \begin{proposition}\label{tensorcoefficients}
   Under  the  assumptions  of  lemma  \ref{lemmafamilyfiniteisotropies}, 
  there is a natural isomorphism  of  differential  graded  algebras
  $$
  C^*_G(M)\otimes_{C^*_K(M)}C_H^*(M)\rightarrow C^*_{G\times_K H} (M).
  $$
 \end{proposition}
  \begin{proof} 
 Lemma  \ref{lemmanaturalequivalencetensor} gives  a  natural   module  isomorphism  $M(G/P)\otimes_{M(G/K)} M(G/Q)\cong M(G\times_K H/ K)$. 
  As  the differential  graded  algebra $C^*_K(M)$ is concentrated in  degree zero,   the  differential  module  structure  on  $ C^*_{G\times _K H}(M)$,  respectively $M(G/P)\otimes_{M(G/K)} M(G/Q)$, agree  with   the  ring structure  on  the  tensor  product $M(G/P)\otimes_{M(G/K)} M(G/Q)$.  This  finishes  the  proof. 
  \end{proof}

Note that if $X$ is a proper $G$-CW-complex and $Y$ is a proper $H$-CW-complex, the product $X\times Y$ has a natural structure of $(G\times H)$-CW-complex (the cells correspond to product of cells of $X$ and $Y$).  From this structure we can construct a $(G\times_KH)$-CW-complex structure in  $X\times Y$. Given a $(G\times H)$-equivariant cell $e_\lambda=D^n\times (G\times H/P\times Q)$, set
$$
e_{\lambda,t}=D^n\times (G\times_KH/P\times_{\pi_1(P)}Q)$$
 for $ t\in(G\times H/G\times_KH)/(P\times Q/P\times_H Q).$ 
 Notice that ${C}_{*}^{G\times_K H}(X\times Y)$ can be obtained as the composition
$$
\Or_{\calf\times_K\calf'}(G\times_KH)\xrightarrow{i_\sharp}\Or_{\calf\times\calf'}(G\times H)\xrightarrow{{C}_{*}^{G\times H}(X\times Y)}\mathbb{Z}-\CHAINCOMPLEXES 
 $$
where $i_\sharp$ is the map induced by the inclusion 
$$
i:G\times_KH\rightarrow G\times H.
$$

\begin{proposition}\label{eilenberg-zilber}
  There is an isomorphism  of  $\Or_{\calf\times_K\calf'}(G\times_K H)$-chain complexes 
  $${C}_{*}^{G\times_K H}(X\times Y)\cong 
  i_\sharp({C}_*^G(X)\otimes{C}_*^H(Y)).
  $$
  Moreover, the isomorphism is  compatible  with the  Differential  Graded  Algebra  structure. 
  \end{proposition}
  
  \begin{proof}
 The identification  of  the isotropy  groups  of the  second  part of  Lemma \ref{lemmafamilyfiniteisotropies}, and  the  usual  Eilenberg-Zilber argument  identify  up  to   chain  homotopy  eqivalence the  chain  complexes over the  orbit  category    $${C}_{*}^{G\times H}(X\times Y)\cong {C}_{*}^G(X)\otimes {C}_{*}^H(Y)$$
as $\Or_{\calf\times\calf'}(G\times H)$-chain complexes.

The differential  graded  structure  is  preserved  since the   differential  graded algebra  $C_K^*(M)$ is  concentrated  in degree  zero and  the  cup  product  agrees  with  the module  structure over  the commutative  ring $M(K/K)$. 

 \end{proof}

We  can  refine  Proposition \ref{pullbackrepresentations} to  an  isomorphism  of  differential  graded  algebras: 

  \begin{proposition} \label{eilenberg-zilber1}
  There is an 
  isomorphism of 
  differential graded algebras
    \begin{align*} 
    \Hom
    (i_\sharp&({C}_*(X)\otimes {C}_*(Y)),
    C^*_G(M)\otimes_{C^*_K(M)}C_H^*(M))\xrightarrow{\cong}&\\& {C}^*_G (X, M)\otimes_{{C}^*_K(M) }{C}^*_H (Y, M).
    \end{align*}

  \end{proposition}
  \begin{proof}
  Notice that in degree $n$ the left hand side cochain complex is
  $$
  \bigoplus_{\lambda,
  \mu}\Hom_\IZ(\IZ[e_\lambda]\otimes_\IZ\IZ[f_\mu],M(G/P_\lambda)
  \otimes_{M(K/\pi_1(P_\lambda))}M(H/Q_\mu)),
  $$
  where $e_\lambda$ denotes a cell in $X$ and $f_\mu$ denotes a cell in $Y$, and 
  the sum is taken over the pairs $\lambda$ and $\mu$ such that 
  $\dim(e_\lambda)+dim(f_\mu)=n$. Note that $\IZ[e_\lambda]\otimes_\IZ\IZ[f_\mu]$ 
  is isomorphic as abelian group to $\IZ$. 
    Then,    each summand in the direct sum is isomorphic to 
    $$
    M(G/P_\lambda)
  \otimes_{M(K/\pi_1(P_\lambda))}M(H/Q_\mu),
  $$
   and the left hand  side cochain complex in degree $n$ 
  is isomorphic to
  $$
  \bigoplus_{\lambda, \mu}M(G/P_\lambda)
  \otimes_{M(K/\pi_1(P_\lambda))}M(H/Q_\mu).
  $$
  The right hand side cochain complex in degree $n$ is 
  $$
  \bigoplus_{\lambda}\Hom_\IZ(\IZ[e_\lambda],M(G/P_\lambda))
  \otimes_{M(K/\pi_1(P_\lambda))}M(H/Q_\mu).
  $$
  Using \ref{eilenberg-zilber},  this term is isomorphic to
  $M(G/P_\lambda)\otimes_{M(K/\pi_1(P_\lambda))}M(H/Q_\mu)$
 and the coboundary maps  are compatible with  the  isomorphism. 
   \end{proof}
  
Recall the  construction  of  the  Eilenberg-Moore   spectral sequence, page  241  in Chapter  7  of  \cite{McCleary2001}. 

  \begin{theorem}\label{spectral}[First Eilenberg-Moore  Theorem]
  Let  $A$ be  a  differential  graded  algebra over  the  ring $R$,  let  $M$  and  $N$  be   
  differential  graded  $A$-modules. Assume  $A$  and the  graded  $R$-Module  of  the  homology of  $A$, $H(A)$  are  flat  modules  
  over $R$. Then,  there  is  a  second  quadrant spectral  sequence  with 
$$
E_2^{p,q} = \Tor _{H(A)}^{p,q}(H(M), H(N)) 
$$
 converging  to  $\Tor_A^{p,q}(M,N).$
  \end{theorem}

Specializing  to  the Bredon  cochain complex  and  the  differential  graded  module  structure,  we  have

  \begin{theorem}\label{eilenbergmoore}
  Let  $\Gamma= G\times_K H$  be  a  group  satisfying  condition 
  \ref{diagrampullback}. Let  $M$  be  a  contravariant  Bredon Functor  taking  values on the  category  of  commutative  rings.   Assume  that  $X$ is  a proper  $G$-CW  complex  and $Y$ is  
  a  proper  $H$-CW  complex. Then, there  is  a  spectral  sequence  with  $E_2$  
  term  given  by  
  $$ 
  \Tor_{H^*(C^*_K(M))}^{p,q}(H_\calf^*(X,M),H_{\calf'}^*(Y,M))  
  $$ 
  which  converges  to 
  $$
  \Tor_{C^*_K(M)}^{p,q}(\underbar{C}_G^*(X,M),
  \underbar{C}_H^*(Y,M)).
  $$ 
  Notice  that, as  the  differential graded  algebra 
  $C^*_K(M)$ is  concentrated  in degree  0  and  has  no  differentials, the  $E_2$  
  term  can be  identified  with  
  $$
  \Tor_{C^*_K(M)}^{p,q}
  (H_\calf^*(X,M),H_{\calf'}^*(Y,M)).
  $$
 \end{theorem}

\begin{proposition}
Denote by  $M$  the Bredon Functor given  by  the  representation  ring. Then, ${C}_G^*(X,M)$ is a $C^*_G(M)$-projective module.
\end{proposition}

\begin{proof}
The  cochain  complex ${C}_G^n(X,M)$ in degree $n$ is  isomorphic  to  a module   of  the  form 
$${\rm }Hom(\bigoplus_\lambda \IZ[e_\lambda], M(?)),$$
where  $e_\lambda$ is  an orbit   cell  of the  type  $G/H_\lambda \times D^n$. 
Using  the  Yoneda  lemma, this  is  isomorphic  to  $\lim_{H_\lambda} M(G/H_\lambda)$,  where the  limit  is  taken with respect  to  $G$-maps  $G/H\to G/K$. 
Since  the  representation  ring  is  semisimple due  to the  Schur-Artin-Wedderburn theorem,  this is  a projective  module  over $C^*_G(M)$.
\end{proof}

 If  conditions \ref{conditionP} are satisfied,  ${C}_G^*(X,M)$ is a projective $C^*_K(M)$-module.
 
  In this case, the spectral sequence of  theorem \ref{eilenbergmoore} collapses at level 2 with
$$\Tor_{C^*_K(M)}^{p,q}(H_\calf^*(X,M),H_{\calf'}^*(Y,M))\cong H_\calf^p(X,M)\otimes_{C^*_K(M)}H_{\calf'}^q(Y,M)),$$
and
$$\Tor_{C^*_K(M)}^{p,q}({C}_G^*(X,M),{C}_H^*(Y,M))\cong H_p({C}_G^*(X,M))\otimes_{C^*_K(M)}{C}_H^q(Y,M)).$$

Proposition \ref{tensorcoefficients}, \ref{eilenberg-zilber}, and \ref{eilenberg-zilber1} yield

  \begin{theorem}[Bredon cohomology of pullbacks]\label{kunneth}
  If conditions \ref{conditionP} are satisfied, there is an isomorphism of   
  $C^*_K(M)$-modules
  $$
  H^*_{\calf\times_K\calf'}(X\times Y,M)\cong  
  H_\calf^*(X,M)\otimes_{C^*_K(M)}H_{\calf'}^*(Y,M).
  $$
\end{theorem}

When we take rational coefficients, the  spectral sequence constructed  above collapses and we obtain a K\"unneth formula.
 Let $M^\IQ(?)$ be the functor $M$ with rational coefficients i.e. $M^\IQ(G/H)=M(G/H)\otimes_\IZ\IQ$.

  \begin{corollary}[Rationalized Bredon cohomology of pullbacks]\label{kunnethrational}
  If conditions \ref{conditionP} are satisfied, there is an isomorphism of   
  $M^\IQ$-modules
$$
H^*_{\calf\times_K\calf'}(X\times Y,M^\IQ)\cong  
  H_\calf^*(X,M^\IQ)\otimes_{M^\IQ(K/K)}H_{\calf'}^*(Y,M^\IQ).$$
  \end{corollary}

In order to apply Theorem \ref{kunneth} in the  following  section, we will  need the following elementary lemma.
 \begin{lemma}\label{restriction of coefficients}
 Let 
 $$
 0\rightarrow A\rightarrow B\rightarrow C\rightarrow0
 $$
 be an exact sequence of projective $R$-modules and $I$be an ideal in $R$, then, the 
  sequence
  $$0\rightarrow A/I\rightarrow B/I\rightarrow C/I\rightarrow0$$
  is exact.
  \end{lemma}

\subsection*{Bredon cohomology  of  Crystallographic   groups of  arbitrary  dimension   with  a  given point  group}

\begin{remark}\label{remarkcrystallographic}
Let  $\Gamma$  be a  group  extension 
$$
1\to \mathbb{Z}^n\to \Gamma \to K \to 1
$$ 
 given  by  the  conjugation  action  of a  representation $\rho: K\to Gl_n(\mathbb{Z})$   of a  finite  group $K$.  
 Then, 

\begin{itemize}
\item  The  space $\mathbb{R}^{n} $ with the  induced  action  is a  model  for  $\eub{\Gamma}$. This  is a consequence  of  Proposition  1.12, page 30 in  \cite{connolykosniewsky}.  
\item Let  $1\to \mathbb{Z}^n\to \Gamma \to K \to 1$ be  a  group  extension  coming  from  a  representation  of  a  finite  group $\rho: K\to Gl_n(\IZ)$. Assume  that $\IZ^n$  with  the  action  given  by  $\rho$  has  a  $K$-invariant  decomposition $\IZ^n[\rho]= A\oplus B$. Denote  by  $G$  the  semidirect  product  $A\rtimes  K$  and  by  $H$  the  semidirect  product  $B\rtimes  K$. Then, the group $\Gamma$  is  isomorphic  to  the  pullback  $G\times_K H$, as it  can  be readily  seen  from  the  following diagram,  
$$\xymatrix{\Gamma \ar[r] \ar[d]& G \ar[d] \\ H \ar[r]& K }   $$

Here,   the  maps $\Gamma\to H$  and $\Gamma \to G$  are  determined  by  the  projections  onto  the  invariant  $K$-submodules $\IZ^n[\rho]\to  A$ and  $\IZ^ n[\rho]\to B$,  which  in  turn induce    group  homomorphisms $\Gamma= \IZ^n \rtimes  K \to G=A\rtimes K$, $\Gamma= \IZ^n \rtimes  K \to G=B\rtimes K$  giving  the relevant  group  homomorphisms  out  of  $\Gamma$.

 \end{itemize}
 
  The  spectral sequence constructed  in Theorem \ref{eilenbergmoore}  suggests a  method  to   compute  the Bredon  cohomology  groups $H_{\Gamma}^*(\eub{\Gamma}, M)$: 
 \begin{itemize}
\item Decompose  the representation  $\rho$  as  direct  sum $\rho= \oplus  n_i \rho_i$  of indecomposable  representations $\rho_i:K\to Gl_{n_i}(\IZ)$ .
\item Consider  the group  extensions 
$$1\to \mathbb{Z}^{n_i} \to \Gamma_i  \to  K\to  1   $$
\item Compute  the  (potentially easier)  Bredon  cohomology  groups  $H_{\Gamma_i}^*(\eub{\Gamma_i}, M)$
\item  Feed  the  spectral  sequence \ref{eilenbergmoore}  with the cohomology  groups.
\item Establish the  relevant  differential graded  module  structures and  obtain information about $H_{\Gamma}^*(\eub{\Gamma}, M)$.

 \end{itemize}
 \end{remark}
For   finite  groups $K$ for  which  any  prime  $p$, the  $p$-Sylow  subgroup  is  of  order less  than $p^3$, there  is  a  finite  number   of  irreducible  such representations $\rho_i$ \cite{hellerreiner}.

We  will  specialize in  crystallographic  groups  with point  group $\mathbb{Z}/4\IZ$ for  an  specific  example and carry out  this  program obtainting  complete  integral  information  in the  next  section.


\section{Computations  for   the Vafa-Witten group $\Gamma$.   }
Consider the action of $\IZ/4\IZ$ on $\IZ^{\oplus6}$ induced from the 
  action of $\IZ/4\IZ$ on $\IC^3$, given by
  $$
  k(z_1,z_2,z_3)=(-z_1,iz_2,iz_3).
  $$
  The associated semidirect product 
  \begin{align*}
  1\to \mathbb{Z}^6 \to \Gamma \to \IZ/4\IZ\to 1
  \end{align*}
  will  be  called  the  Vafa-Witten  Group  and splits as a multiple pullback
  \begin{equation}\label{multiplepullback}
  (\IZ\rtimes\IZ/4\IZ)\times_{\IZ/4\IZ}(\IZ\rtimes\IZ/4\IZ)\times_{\IZ/4\IZ}(\IZ^2\rtimes\IZ/4\IZ)\times_{\IZ/4\IZ}(\IZ^2\rtimes\IZ/4\IZ).
\end{equation}
The first two semidirect products are taken with respect to  the  conjugation action of $\mathbb{Z}/4$  on  $\mathbb{Z}^2$  given  by  scalar  multiplication with  $-1$,  and the two  last  ones  are  given  by  the conjugation action  on $\mathbb{Z}^2$ given by complex  multiplication by   $i$.

First we  will  apply  the  Spectral  sequence  constructed  in previous  sections  to  compute  the  equivariant  $K$-Theory  and  $K$-homology  of  the  classifying  space  $\eub{\Gamma}$.  Using  the  universal  coefficient  Theorem  1.13  in  \cite{BAVE2013},   this  gives  the   equivariant $K$-homology  groups  relevant  to  the  Baum-Connes  conjecture.

Finally,  we  classify the  virtually  cyclic  subgroups  appearing  in $\Gamma$, and  using  results  on the  algebraic  $K$-theory  in  degrees  lower  than  $-1 $,  we  wil  conclude  the  vanishing  result.
  
\subsection*{Topological  $K$-theory and  $K$-homology }
 We  begin with a  recollection  of the   building  blocks  of   the  action,  as  well as  their  Bredon Cohomology  groups.

\subsubsection*{$\IR$ with the action of $\IZ\rtimes\IZ/4\IZ$}

Let $X=\IR$ with the action of the group $G=\IZ\rtimes\IZ/4\IZ$ where the semidirect product is taken with respect to the  action  given  by  multiplication  with  $-1$,  $-1:\IZ/4\IZ\rightarrow \IZ$. Notice that $X$ is a model for $\underbar{E}G$.

The space $X$ has a $G$-CW-complex structure with two $0$-cell  orbits  $\{0,1/2\}$ both with isotropy groups isomorphic to $\IZ/4\IZ$ and one $1$-cell orbit  $[0,1/2]$ with isotropy group $\IZ/2\IZ$. The Bredon cellular complex takes the form
$$0\rightarrow R(\IZc)\oplus R(\IZc)\rightarrow R(\IZ/2\IZ)\rightarrow 0.
$$

Where  $R(\IZ/4\IZ)$  is  the  representation  ring  of  the  finite  cyclic  group of  order  $4$,  $(\IZ/4\IZ)$. This is  a  polynomial  algebra  isomorphic to $\IZ[\eta]/ \eta^4-1$,  where  $\eta$  is  the  generator.

The Bredon cohomology groups of $X$ with respect to the family of finite subgroups $\mathcal{FIN}(\IZ\rtimes\IZ/4\IZ)$ (we denote by $\calf$) with coefficients in representations  can be easily calculated from it and they are concentrated in degree 0, with $H^0_\calf(X,\calr)=\IZ^{\oplus6}$.
From the calculations of the cohomology groups of $X$ we know that we have an exact sequence of projective $R(\IZc)$-modules
$$
0\rightarrow H^0_\calf(X,\calr)\rightarrow (R(\IZc))^2\rightarrow R(\IZ/2\IZ)\rightarrow 0.
$$
As all $R(\IZc)$-modules in the  above exact  sequence are projective,  hence  we can apply Lemma \ref{restriction of coefficients} with the ideal $I=\langle \eta^2-1\rangle$ contained in $R(\IZc)=\IZ[\eta]/\langle\eta^4-1\rangle$, obtaining the exact sequence

\begin{align*} 
0\rightarrow H^0_\calf(X,\calr)/I\cdot &H^0_\calf(X,\calr)\rightarrow \\
&(R(\IZc))^2/I\cdot(R(\IZc)^2)\rightarrow R(\IZ/2\IZ)\rightarrow 0.
\end{align*}

From the last exact sequence we obtain 
$$
H^0_\calf(X,\calr)/I\cdot H^0_\calf(X,\calr)\cong\IZ^{\oplus2}.
$$
Finally, if we denote by $J$ the ideal $\langle \eta-1\rangle$, as $X/G$ is path-connected,  
$$
H^0_\calf(X,\calr)/J\cdot H^0_\calf(X,\calr)\cong\IZ.$$

\subsubsection*{$\IR^2$ with the action of $\IZ^2\rtimes\IZ/4\IZ$}

Let $Y=\IR^2$ with the action of the group $H=\IZ^2\rtimes\IZ/4\IZ$ where the semidirect product is taken with respect to the action  $i:\IZ/4\IZ\rightarrow Gl_2(\IZ)$ given  by  multiplication  by  $i$. Note that $Y$ is a model for $\underbar{E}H$.
$Y$ is an $H$-CW-complex with three $0$-cell orbits  $(0,0)$, $(1/2,0)$ and $(1/2,1/2)$, two $1$cell orbits  $a_0$ and $a_1$ and one $2$-cell orbit  $T$.

The Bredon cellular complex takes the form
$$
0\rightarrow R(\IZc)\oplus R(\IZc)\oplus R(\IZ/2\IZ)\rightarrow \IZ\oplus\IZ\rightarrow\IZ\rightarrow 0.$$
 The Bredon cohomology groups of $Y$ respect to the family of finite subgroups $\mathcal{FIN}(\IZ\rtimes\IZ/4\IZ)$ (we denote by $\calf'$) with coefficients in representations are concentrated in degree 0 and 2, with $H^0_{\calf'}(Y,\calr)=\IZ^{\oplus8}$ and $H^2_{\calf'}(Y,\calr)=\IZ$. Note that it is compatible with the results in \cite{luck2005} in  section  2, since  the  group $\IZ^2\rtimes_i \IZ/4\IZ$   satisfies  hypotheses  M  and NM (Lemma  2.2 in page 1648).

From the calculations of the cohomology groups of $Y$ we know that there is an exact sequence of projective $R(\IZc)$-modules
$$0\rightarrow H^0_{\calf'}(Y,\calr)\rightarrow (R(\IZc))^2\oplus R(\IZ/2\IZ)\rightarrow \IZ^2\rightarrow 0.$$
Taking the quotient by $I$ we obtain

\begin{align*}
0\rightarrow H^0_{\calf'}(Y,\calr)/I\cdot &H^0_{\calf'}(Y,\calr)\rightarrow \\
&(R(\IZc))^2/I\cdot(R(\IZc)^2)\oplus R(\IZ/2\IZ)\rightarrow \IZ^2\rightarrow 0.
\end{align*}

From the last exact sequence we obtain  
$$
H^0_{\calf'}(Y,\calr)/I\cdot H^0_{\calf'}(Y,\calr)\cong\IZ^{\oplus 2}.
$$
As $Y/H$ is path-connected  
$$
H^0_{\calf'}(Y,\calr)/J\cdot H^0_{\calf'}(Y,\calr)\cong\IZ.
$$
Lastly, since $H^2_{\calf'}(Y,\calr)\cong\IZ$ we obtain 
$$
H^2_{\calf'}(Y,\calr)/I\cdot H^2_{\calf'}(Y,\calr)\cong H^2_{\calf'}(Y,\calr)/J\cdot H^2_{\calf'}(Y,\calr)\cong\IZ.
$$

\subsubsection*{$\IR^6$ with the action of $\IZ^6\rtimes\IZc$}
We proceed to calculate the Bredon cohomology groups of the space $X^2 =X\times X$ with the action of the group $G\times_{\IZ/4\IZ}G$. Theorem \ref{kunneth} gives us an isomorphism
$$
H^0_{\calf\times_{\IZ/4\IZ}\calf}(X^2,\calr)\cong H^0_{\calf}(X,\calr)\otimes_{R(\IZ/4\IZ)}H^0_{\calf}(X,\calr).
$$
From the calculations of the cohomology groups of $X$,  we know that we have an exact sequence of projective $R(\IZc)$-modules
\begin{multline}\label{sequencerank12}
0\rightarrow H^0_\calf(X,\calr)\rightarrow (R(\IZc))^2\rightarrow R(\IZ/2\IZ)\rightarrow 0,
\end{multline}
by tensoring this sequence with  $H^0_\calf(X,\calr)$ we obtain the exact sequence

    \begin{multline}\label{es1}
   H^0_\calf(X,\calr)\otimes_{R(\IZ/4\IZ)}H^0_{\calf}(X,\calr)\rightarrow (H^0_\calf(X,\calr))^2\rightarrow\\
    H^0_\calf(X,\calr)/I
   \cdot H^0_\calf(X,\calr)\rightarrow 0.
   \end{multline}
 
The rank of $(H^0_\calf(X,\calr))^2$ is 12, as  can  be seen  from  counting  ranks  in  sequence \ref{sequencerank12}

 and the rank of $H^0_\calf(X,\calr)/I\cdot H^0_\calf(X,\calr)$ is 2 then
$$
H^0_\calf(X,\calr)\otimes_{R(\IZ/4\IZ)}H^0_{\calf}(X,\calr)\cong\IZ^{\oplus10}.
$$
For simplicity we denote this group by $A$.
Dividing the sequence \ref{es1} by the ideal $I$, respectively, by $J$ we obtain
$A/I\cdot A\cong\IZ^{\oplus2}$ and  $A/J\cdot A\cong\IZ$.

We continue calculating
$$
H^0_\calf(X,\calr)\otimes_{R(\IZ/4\IZ)}H^0_{\calf}(X,\calr)\otimes_{R(\IZ/4\IZ)}H^0_{\calf'}(Y,\calr).
$$
From the calculations of the cohomology groups of $Y$ we know that there is  an exact sequence of projective $R(\IZc)$-modules
\begin{equation}
\label{es2}0\rightarrow H^0_{\calf'}(Y,\calr)\rightarrow (R(\IZc))^2\oplus R(\IZ/2\IZ)\rightarrow \IZ^2\rightarrow 0.
\end{equation}
by tensoring  this sequence with  $A= H^0_\calf(X,\calr)\otimes_{R(\IZ/4\IZ)}H^0_{\calf}(X,\calr)   $,  we obtain
\begin{equation}
\label{es3} A \otimes_{R(\IZ/4\IZ)}H^0_{\calf'}(Y,\calr)\rightarrow A^2\oplus (A/I\cdot A)\rightarrow (A/J\cdot A)^2\rightarrow 0,
\end{equation}
then
 $$
 H^0_\calf(X,\calr)\otimes_{R(\IZ/4\IZ)}H^0_{\calf}(X,\calr)\otimes_{R(\IZ/4\IZ)}H^0_{\calf'}(Y,\calr)\cong\IZ^{\oplus20}.
 $$
If we divide \ref{es3} by $I$, respectively by $J$, we obtain 
$$
A\otimes_{R(\IZc)} H^0_{\calf'}(Y,\calr)/I\cdot (A\otimes_{R(\IZc)}  H^0_{\calf'}(Y,\calr))\cong\IZ^{\oplus6}$$ 
and 
$$
A\otimes_{R(\IZc)} H^0_{\calf'}(Y,\calr)/J\cdot (A\otimes_{R(\IZc)}  H^0_{\calf'}(Y,\calr))\cong\IZ.
$$
On the other hand 
$$
H^0_\calf(X,\calr)\otimes_{R(\IZ/4\IZ)}H^0_{\calf}(X,\calr)\otimes_{R(\IZ/4\IZ)}H^2_{\calf'}(Y,\calr)\cong\
A/I\cdot A\cong\IZ.$$
We continue calculating
$$
H^0_\calf(X,\calr)\otimes_{R(\IZ/4\IZ)}H^0_{\calf}(X,\calr)\otimes_{R(\IZ/4\IZ)}H^0_{\calf'}(Y,\calr)\otimes_{R(\IZ/4\IZ)}H^0_{\calf'}(Y,\calr).
$$
Tensoring the sequence \ref{es2} with $A\otimes_{R(\IZc)}H^0_{\calf'}(Y,\calr)$ we obtain 
$$
H^0_\calf(X,\calr)\otimes_{R(\IZ/4\IZ)}H^0_{\calf}(X,\calr)\otimes_{R(\IZ/4\IZ)}H^0_{\calf'}(Y,\calr)\otimes_{R(\IZ/4\IZ)}H^0_{\calf'}(Y,\calr)\cong\IZ^{\oplus44}.
$$
Now,  Theorem \ref{kunneth} implies 
\begin{equation}\label{H0}
H^0_{\calfin(\Gamma)}(\IR^6,\calr)\cong\IZ^{\oplus44}
\end{equation}
On the other hand, Theorem \ref{kunneth} gives us an isomorphism
\begin{align*}
&H^2_{\calfin(\Gamma)}(\IR^6,\calr)\cong\\
&(H^0_\calf(X,\calr)\otimes_{R(\IZ/4\IZ)}H^0_{\calf}(X,\calr)\otimes_{R(\IZ/4\IZ)}H^0_{\calf'}(Y,\calr)\otimes_{R(\IZ/4\IZ)}H^2_{\calf'}(Y,\calr))^2.
\end{align*}

But, 
\begin{align*}
&H^0_\calf(X,\calr)\otimes_{R(\IZ/4\IZ)}H^0_{\calf}(X,\calr)\otimes_{R(\IZ/4\IZ)}H^0_{\calf'}(Y,\calr)\otimes_{R(\IZ/4\IZ)}H^2_{\calf'}(Y,\calr)\cong\\
&\cong A/J\cdot A\cong\IZ.
\end{align*}

Then,
\begin{equation}\label{H2}
H^2_{\calfin(\Gamma)}(\IR^6,\calr)\cong\IZ\oplus\IZ.
\end{equation}

Finally, Theorem \ref{kunneth} gives us an isomorphism

\begin{align}  \label{H4}
&H^4_{\calfin(\Gamma)}(\IR^6,\calr)\cong 
&H^0_\calf(X,\calr)\otimes_{R(\IZ/4\IZ)}H^0_{\calf}(X,\calr)\otimes_{R(\IZ/4\IZ)}H^2_{\calf'}(Y,\calr)\otimes_{R(\IZ/4\IZ)}H^2_{\calf'}(Y,\calr).
\end{align}

But
\begin{align*}
H^0_\calf(X,\calr)\otimes_{R(\IZ/4\IZ)}&H^0_{\calf}(X,\calr)\otimes_{R(\IZ/4\IZ)}H^2_{\calf'}(Y,\calr)\\
&\otimes_{R(\IZ/4\IZ)}H^2_{\calf'}(Y,\calr)\cong A/J\cdot A\cong\IZ.
 \end{align*}

We summarize these results in
  \begin{theorem}
  Let  $\Gamma$  be  the  Vafa-Witten  Group  $\IZ^6\rtimes\IZc$ \ref{vafawittengroup} acting on $\IR^6$ as is described in Section \ref{Introduction}. The  Bredon  cohomology groups  of $\IR^6$  are  given as  follows: 
    \begin{itemize}
    \item $H^0_{\calfin(\Gamma)}(\IR^6,\calr)\cong\IZ^{\oplus44},$
    \item $H^2_{\calfin(\Gamma)}(\IR^6,\calr)\cong\IZ\oplus\IZ,$
    \item $H^4_{\calfin(\Gamma)}(\IR^6,\calr)\cong\IZ,$ and
    \item $H^k_{\calfin(\Gamma)}(\IR^6,\calr)=0\text{, for }k\neq0,2,4.$
    \end{itemize}
  \end{theorem}
\begin{proof}
Recall  the  multiple  pullback  structure 
\begin{equation*}\label{multiplepullback}
  (\IZ\rtimes\IZ/4\IZ)\times_{\IZ/4\IZ}(\IZ\rtimes\IZ/4\IZ)\times_{\IZ/4\IZ}(\IZ^2\rtimes\IZ/4\IZ)\times_{\IZ/4\IZ}(\IZ^2\rtimes\IZ/4\IZ).
\end{equation*}
The  result  can  be   obtained  from  \ref{H0}, \ref{H2}, \ref{H4}.  
\end{proof}
As the Bredon cohomology groups are concentrated in even degrees the Atiyah-Hirzebruch spectral sequence collapses at the  $E_2$  term, and we get

\begin{theorem}[Equivariant K  theory  of  $\eub{\Gamma}$]\label{cristalograficos}
  Let $\Gamma$ denote the group $\IZ^6\rtimes\IZc$ acting on  the  model for  $\eub{\Gamma}$given by  $\IR^6$ as it is described in Section \ref{Introduction}. The equivariant $K$-theory  groups satisfy
    \begin{itemize}
    \item $K^0_{\Gamma}(\underbar{E}\Gamma)\cong\IZ^{\oplus47}$ and
    \item $K^1_{\Gamma}(\underbar{E}\Gamma)=0$
    \end{itemize}
  \end{theorem}

Recall the  universal  coefficient  theorem  for Bredon  cohomology  with  coefficients  in  complex  representations, Theorem 1.13 in \cite{BAVE2013}, which  we  quote  here  for  the  sake  of  completeness: 
\begin{theoremn}[Universal Coefficient Theorem for Bredon Cohomology]
Let  $X$  be  a  proper,  finite    $G$-CW complex.  Let  $M^ ?$ and  $M_?$  be the  complex representation  ring  with contravariant,  respectively  covariant   functoriality.     Then, there  exists   a short  exact  sequence  of  abelian  groups involving  Bredon  Homology  with  coefficients  in $M^?$  and  Bredon  homology  with  coefficients  in $M_?$

 $$ 0\to {Ext}_{\mathbb{Z  }}  (H_{n-1}^{G} (X, M_?), \mathbb{Z})\to  H^ {n}_{G}(X,  M^?) \to {Hom}_{\mathbb{Z  }}  (H_{n}^{G } (X, M_?), \mathbb{Z}) \to  0 $$ 

\end{theoremn}
 
  We  conclude  that  the Bredon homology groups  above are  isomorphic  on  each  degree  to  the equivariant Bredon cohomology  groups.

Now  consider  the  Atiyah  Hirzebruch  spectral  sequence  for  computing  Equivariant  $K$-homology groups. The  $E_2$  term consists  of  the  Bredon  cohomology  groups  $H_{*}^{G} (X, M_?)= H^{*}_{G} (X, M^?)$,  which  are  concentrated  on  even  degrees. Since  all  differentials  in the (homological!)  Atiyah-Hirzebruch  spectral  sequence   are  zero, the  edge  homomorphism identifies  the  zeroth  equivariant $K$-homology   group  with  the  sum $\bigoplus_{N=0,1,2} H_{2n}^{G} (X, M_?)$  and  the  first equivariant  $K$-homology  group  with  $0$.

 On the  other  hand,  the  Baum-Connes  assembly  map $K_* ^{\Gamma}(\underbar{E}\Gamma)\to K_*(C_r^*(\Gamma))$ is  an  isomorphism due  to results  of  Higson-Kasparov \cite{higson-kasparov}. Putting  all this  together,  we  obtain the  following  computation  of  the  reduced  $C^*$-algebra  of  the  group $\Gamma$. 

\begin{corollary}[Equivariant K-Homology of  $\eub{\Gamma}$]
  Let $\Gamma$ denote the group $\IZ^6\rtimes\IZc$ acting on  the  model for  $\eub{\Gamma}$given by  $\IR^6$ as is described in Section \ref{Introduction}.
    \begin{itemize}
    \item $K_0 ^{\Gamma}(\underbar{E}\Gamma)\cong  K_0^*(C_r^*(\Gamma))\cong\IZ^{\oplus47}$ and
    \item $K_1^{\Gamma}(\underbar{E}\Gamma) \cong K_1^*(C_r^*(\Gamma))=0$
    \end{itemize}
  \end{corollary}

\subsection*{Negative  Algebraic K-Theory}\label{subsectionalgebraicktheory}
The success  of  the  Eilenberg-Moore  method  in the  previous  computations  of  Bredon  cohomology  with  respect  to the family  of  finite  subgroups  relies  on  the  structure  lemma  \ref{lemmafamilyfiniteisotropies}. For  the  family  of  virtually  cyclic  subgroups, there  is  no  such  decomposition. The  following  result,  however, identifies  some  restrictions  for  a  subgroup  in $\Gamma$  in order  to  be  virtually  cyclic.

\begin{proposition}\label{pullbackvirtuallycyclic}
Let  $\Gamma$  be a  group  obtained  as  a  pullback  of  the  type
$$
\xymatrix{ \Gamma \ar[r]_-{p_2} \ar[d]^-{p{_1}} & G \ar[d]^{\pi_1} \\  H \ar[r]_{\pi_2} & K},  
$$ 
 where $p_1$ and $p_2$ are  surjective  maps. 
Given a virtually  cyclic  subgroup $V\leq \Gamma$, the  groups $p_1(V)$, $p_2(V)$ are  virtually  cyclic. 
\end{proposition}

We define the following  family of  subgroups of $\Gamma$  
\begin{multline*}
\mathcal{VC}(G)\times_K \mathcal{VC}(H)=\\ \{V_1\times_{\pi_{1}(V_1)} V_2\mid\,  \text{$V_1\in\mathcal{VC}(G)$ and $V_2\in\mathcal{VC}(H)$}\} 
\end{multline*}

The  family $\mathcal{VC}(G)\times_K \mathcal{VC}(H)$ does  not agree  with the  family  of  virtually cyclic  subgroups  of  $\Gamma$. However,  every virtually cyclic  subgroup  in $\Gamma $ is  contained  in  an element of  the  family.

Thus,  a strategy  to the  classification of  virtually  cyclic  subgroups  of the Group  $\Gamma$  consists  of using  the  iterated  pullback  decomposition  \ref{multiplepullback},  the  several projections  to the  components, as  well as  classification results  for  the  family  of  virtually  cyclic  subgroups  of the  components, take  the pullback  family  and  verify  whether  the  groups appearing there  are virtually  cyclic. 

\begin{proposition}\label{classificationvirtuallycyclic}
The virtually  cyclic  subgroups  of the  Vafa-Witten  Group are, up to isomorphism,  as  follows: 
\begin{itemize}
\item Finite  groups  $0$, $\mathbb{Z}/2\IZ$, $\IZ/4\IZ$,
\item the infinite cyclic group $\IZ$,
\item $\IZ/2\IZ \times \IZ$, $\IZ/4\IZ \times \IZ$, and
\item $\IZ/4\IZ \ast_{\IZ/2\IZ} \IZ/4\IZ$, $D_\infty$, $D_\infty\times \IZ/2\IZ$, $D_\infty\times \IZ/4\IZ$. 
\end{itemize}
\end{proposition}
\begin{proof}
The finite  groups are  readily  realizable. 
We obtain the groups  $\IZ/2\IZ \times \IZ$, $\IZ/4\IZ \times \IZ$ inside  the  product $\IZ^4 \rtimes_{-1} \mathbb{Z}/4\IZ$

On the  other  hand, the  infinite virtually  cyclic  subgroups  of the  group $\IZ^2\rtimes_i \IZ/4\IZ$  have  been  classified in  Lemma 3.7,  page 1656 of \cite{luck2005}, which are either  cyclic or $D_\infty$. 

From the  product  family for the  pullback 
$$
\IZ^2\rtimes_i \IZ/4\IZ \times_{\IZ/4\IZ} \IZ^4 \rtimes_{-1} \mathbb{Z}/4\IZ,
$$
 and from  the  group  
 $$
 \IZ \rtimes_{-1} \IZ/4\IZ\cong\IZ/4\IZ \ast_{\IZ/2\IZ} \IZ/4\IZ 
 $$
 we obtain  the  virtually  cyclic  subgroups 
 $$
 D_\infty, \quad D_\infty\times  \IZ/2\IZ,\quad D_\infty\times  \IZ/4\IZ,
 $$ and  the  group $\IZ\rtimes_{-1} \IZ/4\IZ$, which  is  isomorphic  to  the  amalgam 
 $$\IZ/4\IZ \ast_{\IZ/2\IZ} \IZ/4\IZ .
 $$
\end{proof}

The validity of the Farrell-Jones isomorphism for $\Gamma$ is a well established fact, see for example \cite{luck-stamm} this means that 
the \textit{assembly map} in \ref{equationfarrelljones} is an isomorphism. Thus, the algebraic $K$-theory groups of $R[\Gamma]$ are isomorphic to
the equivariant homology groups
$$
\IH^\Gamma_i(E_{\mathcal{VC}}(\Gamma),\IK^{-\infty}(R))\text{ for all } i\in \IZ.
$$
In order to compute these groups there is an Atiyah-Hirzebruch, \cite{luck1998} spectral sequence converging to them  with second page given by
$$
E^2_{p,q}\cong H_p(B_{\mathcal{VC}}; \{\mathcal{K}_q(R[V])\}),
$$
where the above are homology groups with local coefficients in the algebraic $K$-theory groups of $R[V]$ and $V$ in the family of virtually cyclic subgroups of $\Gamma$.
Let us  analyze the coefficients in the above homology groups $ H_p(B_{\mathcal{VC}}; \{\mathcal{K}_q(R[V])\})$ for $p+q\leq -1$. 
First observe that
from the work of Carter \cite[Theorem 1]{carterlower}, we have that the group $K_{-1}(\IZ[G])$ vanishes for the finite groups of our list above and by \cite[Theorem 3]{carterlocalization}, $K_{-i}(\IZ[G])=0$ for all $i>1$ and all   finite groups $G$.  
By the work of T. Farrell and L. Jones \cite[Theorem 2.1 (b)]{farrelvirtuallycyclic} and the generalizations in \cite{Dj}, $K_{-1}(R[V])$ vanishes for $V$ infinite virtually cyclic  subgroup of our list above and by \cite[Theorem 2.1 (a)]{farrelvirtuallycyclic} and the generalizations in \cite{Dj}, $K_{-i}(R[V])$ also vanish for $i\geq 2$ and for all virtually cyclic groups $V$.  Hence the above spectral sequences consists of zero terms in the range $p+q\leq-1$. Hence we have:

\begin{theorem}\label{theoremnegative}
Let  $R$  be a  ring of  algebraic  integers. Let  $i\leq -1$.  Then,  the  algebraic  $K$-Theory  groups  $K_i(R\Gamma)$  vanish. 
\end{theorem}

\bibliographystyle{alpha}
\bibliography{BAJUVE29AGO}
\end{document}